\documentclass{amsproc}
\usepackage{amssymb,amsmath,amsthm,latexsym,color}
\usepackage{amsfonts}\usepackage{lscape}

\newtheorem{theorem}{Theorem}[section]

\newtheorem{prop}[theorem]{Proposition}
\newtheorem{cor}[theorem]{Corollary}

\theoremstyle{definition}

\theoremstyle{remark}
\newtheorem{remark}[theorem]{Remark}

\numberwithin{equation}{section}

\def\B{\mathbb B}

\def\cC{\mathcal C}

\def\cP{{\mathcal P}}

\def\cS{\mathcal S}

\def\cN{\mathcal N}

\def\cK{\mathcal K}

\def\cA{\mathcal A}
\def\cB{\mathcal B}
\def\cD{\mathcal D}
\def\cG{\mathcal G}
\def\cCG{\mathcal {CG}}
\def\cCM{\mathcal {CM}}
\def\cDY{\mathcal {DY}}
\def\cBLP{\mathcal {BLP}}
\def\cCHK{\mathcal {CHK}}
\def\cBH{\mathcal {BH}}
\def\cZKW{\mathcal {ZKW}}
\def\cPW{\mathcal {PW}}

\def\K{\mathbb K}
\def\F{\mathbb F}
\def\S{\mathbb S}
\def\V{\mathbb V}
\def\N{\mathbb N}

\def\S{\mathbb S}

\def\cP{{\mathcal LMPTB}}
\def\bB{\mathbb B}

\def\ps@headings{
 \def\@oddhead{\footnotesize\rm\hfill\runningheadodd\hfill\thepage}
 \def\@evenhead{\footnotesize\rm\thepage\hfill\runningheadeven\hfill}
 \def\@oddfoot{}
 \def\@evenfoot{\@oddfoot}
}
\begin{document}

\title{On the nuclei of a finite semifield}

\author{Giuseppe Marino}
\address{Dipartimento di Matematica, Seconda Universit\`a degli Studi di Napoli,
I--\,81100 Caserta, Italy}
\email{giuseppe.marino@unina2.it}
\thanks{This work was
supported by the Research Project of MIUR (Italian Office for
University and Research) ``Geometrie su Campi di Galois, piani di
traslazione e geometrie di incidenza''.}

\author{Olga Polverino}
\address{Dipartimento di Matematica, Seconda Universit\`a degli Studi di Napoli,
I--\,81100 Caserta, Italy}
\email{olga.polverino@unina2.it}

\date{}


\begin{abstract}
In this paper  we collect and improve the techniques for calculating the nuclei of a semifield and
we use these tools to determine the order
 of the nuclei and of the center of some commutative presemifields of odd characteristic
recently constructed.
\end{abstract}

\maketitle

\section{Introduction}
Semifields are algebras satisfying all the axioms for a skewfield
except (possibly) associativity of the multiplication. From a
geometric point of view, semifields coordinatize certain
translation planes ({\em semifield planes}) which are
planes of Lenz--Barlotti class $V$ (see, e.g., \cite[Sec.
5.1]{Dembowski1968}) and, by \cite{Albert1960}, the isomorphism
relation between two semifield planes corresponds to the isotopism
relation between the associated semifields. The first example of a
finite semifield which is not a field was constructed by Dickson
about a century ago in \cite{Dickson1906-1}, using the term {\em
nonassociative division ring}. These examples are commutative
semifields of order $q^{2k}$ and they exist for each $q$ odd prime
power and for each $k>1$ odd. Since then and until 2008, the only
other known families of commutative semifields of odd
characteristic $p$, existing for each value of $p$, have been some
Generalized twisted fields constructed by Albert in
\cite{Albert1961P}.

The relationship between  commutative semifields of odd order and
planar DO polynomials has given new impetus to construct new examples of
such algebraic structures. Indeed in \cite{ZaKyWa2009},
\cite{BuHe2008}, \cite{Bierbrauer2010}, \cite{LuMaPoTr2011},
\cite{BierbrauerSub} and \cite{ZhPoSub}, several families of commutative semifields in
odd characteristic have been constructed.

In this paper  we collect and improve the results of the last
years on techniques for calculating the nuclei of a semifield and
we use these tools to determine the order
 of the nuclei and of the center of the
semifields presented in \cite{ZaKyWa2009}, \cite{BuHe2008} and
\cite{Bierbrauer2010}. From these results  we are able to prove
that, when the order of the center is grater than 3, the
Zha--Kyureghyan--Wang presemifields and the Budaghyan--Helleseth presemifields of \cite{BierbrauerSub} are
new. Precisely, each Zha--Kyureghyan--Wang presemifield  \cite{ZaKyWa2009} is not
isotopic to any Budaghyan--Helleseth presemifield \cite{BuHe2008} and both of them
are not isotopic to any previously known presemifield. Also, using the same arguments we show that the Bierbrauer presemifields \cite{Bierbrauer2010} are isotopic to neither a Dickson semifield, nor to a Generalized twisted field and to any of the known presemifields in characteristic 3.

\section{Isotopy relation and nuclei}
A finite \textit{semifield} $\mathbb{S}=(S,+,\star)$ is a finite
binary algebraic structure satisfying all the axioms for a
skewfield except (possibly) associativity of multiplication. The
subsets of $S$

\[
\N_{l}=\{a\in S\,|\,(a\star b)\star c=a\star(b\star c),\,\forall b,c\in S\},
\]%
\[
\N_{m}=\{b\in S\,|\,(a\star b)\star c=a\star(b\star c),\,\forall a,c\in S\},
\]%
\[
\N_{r}=\{c\in S\,|\,(a\star b)\star c=a\star(b\star c),\,\forall a,b\in S\}
\]
and
\[
\K=\{a\in \N_{l}\cap \N_{m}\cap \N_{r}\,|\,a\star b=b\star a,\,\forall b \in
S\}
\]

are fields and are known, respectively, as the \textit{left
nucleus}, \textit{middle nucleus}, \textit{right nucleus} and
\textit{center} of the semifield. A finite semifield is a vector
space over its nuclei and its center.

If $\mathbb{S}$ satisfies all axioms for a semifield except,
possibly,  the existence of an identity element for the
multiplication then we call it a \textit{presemifield}. The additive
group of a presemifield is an elementary abelian $p$--group, for
some prime $p$ called the {\it characteristic} of $\S$.

Two presemifields, say $\S=(S,+,\star)$ and $\S'=(S',+,\star')$,
with characteristic $p$, are said to be \textit{isotopic} if there
exist three invertible $\F_p$-linear maps $g_1,g_2,g_3$ from $S$ to
$S'$ such that
$$ g_1(x) \star' g_2(y) = g_3(x \star y) $$
for all $x,y \in S$; the triple $(g_1,g_2,g_3)$ is an {\it
isotopism} between $\S$ and $\S'$. In each isotopy class of a
presemifield we can find  semifields (see~\cite[p.
204]{Knuth1965}). The sizes of the nuclei as well as the size of the
center of a semifield are invariant under isotopy; for this reason
we refer to them as the {\it parameters} of $\S$. Whereas, if $\S$
is a presemifield, then the  {\it parameters} of $\S$ will be the
parameters of any semifield isotopic to it. If $\S=(S,+,\star)$ is a
presemifield, then $\S^d=(S,+,\star^d)$, where $x\star^d y=y\star x$ is
a presemifield as well, and it is called the {\em dual} of $\S$. For
a recent overview on the theory of finite semifields see Chapter 6 \cite{LaPo201*} in the collected work \cite{deBeSt201*}.

\bigskip

Let $\S=(S,+,\star)$ be a presemifield  having characteristic $p$
and order $p^t$. The set

$$\cC=\{\varphi_y\,:\, x\in S\, \rightarrow \, x\star y\in S\,|\,
y\in S\} \subset \V={\rm End}_{\F_p}(S)$$

is the {\it semifield spread set} associated with $\S$ ({\it spread
set} for short): $\cC$ is an $\F_p$-subspace of $\V$ of rank $t$
and each non-zero element of $\cC$ is invertible. Also, if $\S$ is a
semifield and $e$ is the identity element of $\S$, then
$id=\varphi_e\in \cC$. It can be seen that, by translating the isotopy
relation between presemifields in terms of the associated spread
sets, just two maps of the triple  $(g_1,g_2,g_3)$ are involved. Indeed

\begin{prop}{\rm \cite[Prop.2.1]{MaPoSub}}\label{prop:IsotopicSemifAndSpreadSets}
Let $\S_1=(S_1,+,\bullet)$ and  $\S_2=(S_2,+,\star)$ be two
presemifields and let $\cC_1$ and $\cC_2$ be the corresponding
spread sets. Then $\S_1$ and $\S_2$ are isotopic under the
isotopism $(g_1,g_2,g_3)$ if and only if
$\cC_2=g_3\cC_1g_1^{-1}=\{g_3\circ \varphi_{y}\circ g_1^{-1}|\
y\in S_1\}$\footnote{Here "$\circ$" stands for composition of
maps.}.
\end{prop}
\begin{proof} Let $\cC_1=\{\varphi_y\,|\, y\in S_1\}$ and  $\cC_2=\{\varphi'_y\,|\, y\in S_2\}$. The necessary condition can be easily proven. Indeed, if
$(g_1,g_2,g_3)$ is an isotopism between $\S_1$ and $\S_2$, then
$g_3(\varphi_y(x))=\varphi'_{g_2(y)}(g_1(x))$ for each $x,y\in S_1$.
Hence, $\varphi'_{g_2(y)}=g_3\circ \varphi_y\circ g_1^{-1}$ for each
$y\in S_1$ and the statement follows taking into account that
$\cC_2=\{\varphi'_y\,|\, y\in S_2\}=\{\varphi'_{g_2(y)}|\ y\in
S_1\}$.

Conversely, suppose that $\cC_2=\{g_3\circ \varphi_{y}\circ
g_1^{-1}|\ y\in S_1\}$, where $g_1$ and $g_3$ are  invertible
$\F_p$--linear maps from $S_1$ to $S_2$. Then the map $g_2$,
sending each element $y\in S_1$ to the unique element $z\in S_2$
such that $\varphi'_{z}=g_3\circ \varphi_y\circ g_1^{-1}$ (where
$\varphi_z'\in \cC_2$), is an invertible $\F_p$--linear map
 from $S_1$ to $S_2$. Hence, for each $x,y\in S_1$ we get
$\varphi'_{g_2(y)}(x)=g_3(\varphi_{y}(g_1^{-1}(x)))$, i.e. $x\star
g_2(y)=g_3(g_1^{-1}(x)\bullet y)$ and putting $x'=g_1^{-1}(x)$ we
have the assertion.
\end{proof}

Move the study  from the presemifield  to the associated  spread
set, allows  to determine its  parameters without passing
through an isotopic semifield.

The following result generalizes \cite[Thm. 2.1]{MaPoTr2011}.

\begin{theorem} \label{thm:nucleipresemifield}
Let $\S = (S,+,\star)$ be a presemifield of characteristic $p$ and
let $\cC$ be the associated spread set of $\F_p$--linear maps. Then
\begin{enumerate}
\item the right  nucleus  of each semifield isotopic to $\S$ is
isomorphic to the largest field $\cN_r(\S)$ contained in
$\V={\rm End}_{\F_p}(S)$ such that $\cN_r(\S)\cC\subseteq \cC$;
 \item the
middle nucleus  of each semifield isotopic to $\S$ is isomorphic
to the largest field $\cN_m(\S)$ contained in $\V$
such that $\cC\cN_m(\S)\subseteq \cC$;
 \item the left nucleus of each semifield isotopic to $\S$ is isomorphic to the
largest field $\cN_l(\S)$ contained in $\V$ such that
$\cN_l(\S)\cC^*\subseteq \cC^*$, where  $\cC^*$ is the spread set
associated with the dual presemifield $\S^*$ of $\S$;
 \item  the center of each semifield isotopic to $\S$ is isomorphic
to the largest field $\cK_{r,\omega}(\S)$ contained in $\cN_r(\S)$
such that
\begin{equation} \label{eqRiCenter}
\rho \circ \varphi= \varphi \circ (\omega^{-1}\circ \rho \circ
\omega)
\end{equation}
 for
all $\rho \in \cK_{r,\omega}(\S)$ and  $\varphi\in \cC$, where
$\omega$ is a fixed invertible element of $\cC$. Equivalently, the
center of each semifield isotopic to $\S$ is isomorphic to the
largest field $\cK_{m, \sigma}(\S)$ contained in $\cN_m(\S)$ such
that
\begin{equation} \label{eqMiCenter}
\varphi \circ \rho=  (\sigma^{-1}\circ \rho \circ \sigma)\circ
\varphi
\end{equation}
 for all $\rho \in \cK_{m,\sigma}(\S)$ and  $\varphi\in \cC$, where
$\sigma$ is a fixed invertible element of $\cC$. Also,
$\cK_{m,\sigma}(\S)$ and $\cK_{r,\omega}(\S)$ are conjugated fields.
\end{enumerate}
\end{theorem}
\begin{proof}
Let $\S'=(S',+,\bullet)$ be a semifield isotopic to $\S$ and let
$\cC$ and $\cC'$ be the associated spread sets. Then by  the
previous proposition, $\cC=g_3\cC' g_1^{-1}$ for some invertible
$\F_p$-linear maps from $S'$ to $S$.  If $\N_r(\S')$ is the right
nucleus of $\S'$, then it easy to see that $\cN_r(\S')=\{\varphi_y
\colon y \in \N_r(\S') \}$ is the maximum field contained in  $\V$
such that $\cN_r(\S')\cC'\subseteq \cC'$, i.e. for each
$\mu\in\cN(\S')$ we have $\mu\circ\varphi'_y\in \cC'$ for every
$\varphi'_y\in \cC'$, and, obviously, $\cN_r(\S')$ is isomorphic to
$\N_r(\S')$. Then, we have that $\cN_r(\S')^{g_3^{-1}}:
=g_3\cN_r(\S')g_3^{-1}$ is the maximum field contained in $\V$ with
respect to which $\cN_r(\S')^{g_3^{-1}} \cC\subseteq \cC$, i.e.
$\cN_r(\S)=\cN_r(\S')^{g_3^{-1}}$, and, clearly, $\cN_r(\S)$ is
isomorphic to $\N_r(\S')$. This shows our claim in Case $(a)$. The
same arguments can be used to prove point $(b)$; in such a case we get that $\cN_m(\S)=\cN_m(\S')^{g_1^{-1}}$ .\\
Now, noting that  $\N_l(\S')=\N_r(\S'^d)$, applying point $(a)$ to $\S^d$, we get point $(c)$.\\
Finally, let $\K(\S')$ be the center of $\S'$ and note that
$\cK(\S')=\{\varphi'_y \colon y \in \K(\S') \}$ can be seen as the
maximum subfield contained in  $\cN_r(\S')$ (or contained in
$\cN_m(\S')$) such that $\mu \circ \varphi'= \varphi' \circ \mu$ for
each $\varphi'\in \cC'$  and, obviously, $\cK(\S')$ is isomorphic to
$\K(\S')$. Now, since $\cC=g_3\cC' g_1^{-1}$,
$\cN_r(\S)=\cN_r(\S')^{g_3^{-1}}$ and
$\cN_m(\S)=\cN_m(\S')^{g_1^{-1}}$, we have that for each element
$\rho\in \cK(\S')^{g_3^{-1}}$ and for each $\varphi \in \cC$
$$\rho\circ \varphi=\varphi\circ \rho^\omega$$
 where
$\omega=g_3\circ g_1^{-1}$. Note that, since $\S'$ is a semifield,
$\omega$ is an element of $\cC$. Now, it is easy to see that
$\cK(\S')^{g_3^{-1}}$ is the maximum subfield $\cK_{r,\omega}(\S)$
contained in $\cN_r(\S)$ satisfying (\ref{eqRiCenter}). In the same
way, we have hat $\cK(\S')^{g_1^{-1}}$ is the maximum subfield
$\cK_{m,\sigma}(\S)$ contained in $\cN_m(\S)$ satisfying
(\ref{eqMiCenter}).
\end{proof}

By the previous result it follows that we can define the middle
nucleus (resp. right nucleus) of a presemifield $\S=(S,+,\star)$,
with associated spread set $\cC$, as the largest field contained
in $\V={\rm End}_{\F_p}(S)$ with respect to which $\cC$ is a right
vector space (resp. left vector space); similarly, the left
nucleus of $\S$
 can be defined  as the largest field contained in
$\V$ with respect to which $\cC^d$ (the spread set associated
with the dual of $\S$) is a left vector space. Also, regarding the center, note
that if $\S$ is a semifield, then $id\in \cC$ and hence
$$\cK_{r,id}(\S)=\cK_{m,id}(\S)=\cK(\S)=\{\mu \in \cN_{r}(\S)\cap \cN_{m}(\S)\, :\, \mu \circ \varphi =\varphi \circ \mu \ \forall \varphi \in
\cC\}.$$

\bigskip

By the proof of the previous theorem we also  have the following
result.

\begin{cor} \label{cor:nucleiisotopy}
If the presemifields $\S_1$ and $\S_2$ are isotopic via the isotopy
$(g_1,g_2,g_3)$ then
\begin{enumerate}
\item[1] $\cN_r(\S_2)=g_3 \cN_r(\S_1)g_3^{-1}$;

\item[2] $\cN_m(\S_2)=g_1 \cN_m(\S_1)g_1^{-1}$;

\item[3]  $\cN_l(\S_2)=g_3 \cN_l(\S_1)g_3^{-1}$;

\item[4]  $\cK_{r,\sigma}(\S_2)=g_3 \cK_{r,\omega}(\S_1)g_3^{-1}$ and
$\cK_{m, \sigma}(\S_2)=g_1 \cK_{m,\omega}(\S_1)g_1^{-1}$, where
$\omega \in \cC_{1}\setminus \{0\}$ and $\sigma=g_3\circ \omega\circ
g_1^{-1} \in \cC_2$.

\end{enumerate}
\end{cor}

\subsection{The Knuth Chain}
\vskip.2cm\noindent If $\mathbb S=(S,+,\star)$ is a presemifield  $n$-dimensional  over
${\mathbb F}_p$, and $\{e_1,\ldots, e_n\}$ is an $\F_p$-basis for
$\mathbb S$, then the multiplication can be written via the
multiplication of the vectors  $e_i$'s. Indeed, if $x=x_1e_1+\cdots
+ x_ne_n$ and $y=y_1e_1+\cdots + y_ne_n$, with $x_i,y_i\in {\F_p}$,
then
\begin{eqnarray}
x\star y =\sum_{i,j=1}^{n}x_iy_j(e_i\star e_j )=
\sum_{i,j=1}^{n}x_iy_j \left (\sum_{k=1}^n a_{ijk} e_k\right )
\end{eqnarray}
for certain $a_{ijk} \in {\F_p}$, called the {\it structure
constants} of $\mathbb S$ with respect to the basis $\{e_1,\ldots,
e_n\}$. Knuth noted, in \cite{Knuth1965}, that the action of the
symmetric group $\cS_3$ on the indices of the structure constants
gives rise to another five presemifields starting from one
presemifield $\mathbb S$. The set $[{\mathbb S}]$ of these  (at most
six) presemifields is called the {\it Knuth Chain} of $\S$, and consists of the presemifields $\{{\mathbb
S},{\mathbb S}^{(12)},{\mathbb S}^{(13)},{\mathbb
S}^{(23)},{\mathbb S}^{(123)},{\mathbb S}^{(132)}\}$, called the {\em derivatives} of $\S$ ($\S$ included).\\
In the same paper, the author proved that the action of $\cS_3$ on the indices of the structure constants of a
presemifield $\mathbb S$ is well-defined with respect to the
isotopism classes of $\mathbb S$, and by the {\it Knuth orbit of
$\mathbb S$} we mean the set of isotopism classes
corresponding to the Knuth chain $\mathbb S$.\\
The presemifield ${\mathbb S}^{(12)}$ is the {\it opposite
algebra} of ${\mathbb{S}}$ obtained by reversing the
multiplication, or in other words,
${\mathbb{S}}^{(12)}={\mathbb{S}}^{d}$, the dual of $\S$.
Similarly, it is easy to see that the semifield
${\mathbb{S}}^{(23)}$ can be obtained by transposing the matrices
corresponding to the transformations $\varphi_y$, $y \in \S$, with
respect to some $\F_p$-basis of $\S$, and for this reason
${\mathbb{S}}^{(23)}$ is also denoted by ${\mathbb{S}}^t$, called
the {\it transpose of ${\mathbb{S}}$}. With this notation, the
Knuth orbit becomes $ \{[{\mathbb S}],[{\mathbb S}^{d}],[{\mathbb
S}^{t}],[{\mathbb S}^{dt}],[{\mathbb S}^{td}],[{\mathbb
S}^{dtd}]\}.$ Note that $t$ and $d$ are operations of order two,
i.e. $({\mathbb S}^{t})^t=\mathbb S$ and $({\mathbb
S}^{d})^d=\mathbb S$.

\bigskip

It is possible to describe the transpose of a presemifield without
fixing a basis of $\S$.\\
Let $\S=(S,+,\star)$ be a presemifield of characteristic $p$ and
order $p^n$, let $\V={\rm End}_{\F_p}(S)$ and let $\cC$ be the associated
spread set. Denote by $\langle, \rangle$ a non-degenerate symmetric
bilinear form  of $S$ as $\F_p$-vector space and   denote by
$\overline \varphi$ the adjoint of   $\varphi \in \V$ with respect
to $\langle, \rangle$, i.e. $\langle x, \varphi(y) \rangle=\langle
{\overline \varphi}(x), y \rangle$ for each $x,y \in S$. Since the
map $T:\, \varphi \in \V \, \mapsto \overline{\varphi} \in \V$ is an
involutive antiautomorphism of the endomorphisms ring $\V$ and $dim
Ker \varphi=dimKer \overline \varphi$, we get that ${\overline \cC}=\{{\overline \varphi_y}\, :\, \varphi_y
\in \cC\}$ is an additive spread set as well,
defining the presemifield $\overline{\S}=(S,+,\overline{\star})$
where $x \overline{\star} y={\overline \varphi_y}(x)$ for each $x,y
\in S$. It is possible to prove that $\overline{\S}$, up to isotopy,
does not depend on the choice of the bilinear form $\langle,
\rangle$ and that $\overline{\S}$ is isotopic to the presemifield
$\S^t$. For this reason, in what follows, fixed a
non-degenerate symmetric bilinear form  $\langle, \rangle$ of $S$,
 the presemifield  $\overline \S$, constructed by using
the adjoints with respect to $\langle, \rangle$,  will be denoted as
$\S^{t}$ and the associated spread set ${\overline \cC}$ will be
denoted as $\cC^t$. Moreover, if $X$ is a subset of $\V$, we will
denote by $\overline X$ the set of the adjoint maps, with respect to
$\langle, \rangle$, of the elements of $X$.

Now, we are able to describe how the nuclei move in the Knuth chain (see also \cite{Maduram1975} and \cite{Lunardon2006}).

\begin{prop} \label{propo:gironuclei}
If $\S$ is a presemifield, then
\begin{enumerate}
\item[1.] $\cN_r(\S)=\cN_l(\S^d)=\overline{\cN_m(\S^t)}$;

\item[2.] $\cN_m(\S)=\overline{\cN_r(\S^t)} \cong \cN_m(\S^d)$;

\item[3.]  $\cN_l(\S)=\cN_r(\S^d)\cong \cN_l(\S^t)$.

\end{enumerate}
\end{prop}

\begin{proof}
Point 1. follows from Theorem \ref{thm:nucleipresemifield} and
from properties of the adjoint maps. Indeed, since $T$ is an
antiautomorphism, we get $T(\varphi\circ \mu)=T(\mu)\circ
T(\varphi)=\bar \mu\circ\bar\varphi$ for each $\mu,\varphi\in\V$.
The first part of $2.$ and $3.$ follows from $1.$. Now, note that
if $\mu$ is an element of $\cN_m(\S)$, then for each $y \in S$
there exists a unique element $z\in S$ such that $\varphi_y\circ
\mu=\varphi_z$ and the map $\sigma_\mu\, : y\in S\, \mapsto \,
z\in S$ is an invertible $\F_p$-linear map of $S$; so, $F=\{\sigma_\mu \,:\, \mu \in \cN_m(\S) \}$ is a field of maps
isomorphic to $\cN_m(\S)$ satisfying $(b)$ of Theorem
\ref{thm:nucleipresemifield} relatively to $\cC^d$, i.e.
$\cN_m(\S)\cong F=\cN_m(\S^d)$. Finally, by using the previous relations and taking into account that $\S^{dtd}=\S^{tdt}$ we get $$\cN_\ell(\S^t)=\cN_\ell(\S^{tdd})=\overline{\cN_m(\S^{tdt})}=\overline{\cN_m(\S^{dtd})}\cong \cN_r(\S^{dtt})= \cN_r(\S^d)=\cN_\ell(\S).$$

\end{proof}

\subsection{Semifields and $q$-polynomials}

If $\S$ is a presemifield of characteristic $p$ and order $p^n$,
then we may assume, up to isomorphisms, that
$\S=(\F_{p^n},+,\star)$, where $x\star y=F(x,y)$. Since $F(x,y)$
is additive with respect to both the variables $x$ and $y$, it can be seen as the polynomial map associated with a
$p$-polynomial of $\F_{p^n}[x,y]$, i.e.
$$F(x,y)=\sum_{i,j=0}^{n-1}a_{i,j}x^{p^i}y^{p^j}$$
where
$a_{ij}\in\F_{p^n}$. Also, each  element $\varphi$ of $\V={\rm End}_{\F_p}(\F_{p^n})$
can be written in a unique way as $\varphi(x)=\sum_{i=0}^{n-1}\beta_ix^{p^i}$.\\
Now, let $\langle, \rangle$ be the symmetric bilinear form of
$\F_{p^n}$ over $\F_p$ defined by the following rule $\langle x,y
\rangle =tr_{p^n/p}(xy)$. Then $\langle, \rangle$ is a
non-degenerate symmetric bilinear form and the adjoint
$\bar\varphi$ of the element
$\varphi(x)=\sum_{i=0}^{n-1}\beta_ix^{p^i}$ of $\V$ with respect to $\langle,\rangle$, is $\bar\varphi(x)=\sum_{i=0}^{n-1}\beta_i^{p^{n-i}}x^{p^{n-i}}.$
This implies that the dual and the transpose  of $\S$ are defined,
respectively,  by the following multiplications
$$x \star^d y= F(y,x)$$
and
$$x \star^t y=\sum_{i,j=0}^{n-1}a_{n-i,j}^{p^i}x^{p^i}y^{p^{i+j}},$$

where the indices $i$ and $j$ are considered modulo $n$.

The polynomial $F$ defining the multiplication of $\S$  can be
useful to determine the order of the nuclei. In what follows, if
$\F_q$ is a subfield of $\F_{p^n}$, then we will denote by $F_q$
the corresponding field of scalar maps $\{{t_\lambda} \,:\,
x\in\F_{p^n}\mapsto\lambda x\in\F_{p^n}|\,\lambda\in\F_q\}$
contained in $\V$.

\begin{theorem} \label{thm:nucleipresemifieldfromF}
Let $\S = (\F_{p^n},+,\star)$ be a presemifield whose
multiplication is given by $x\star y=F(x,y)$, with
$F(x,y)=\sum_{i,j=0}^{n-1}a_{ij}x^{p^i}y^{p^j}$ and
$a_{ij}\in\F_{p^n}$ and let $\F_q$ be a subfield of $\F_{p^n}$.
\begin{itemize} \item[(A)] If there exists
$\tau\in Aut(\F_q)$ such that $$F(\lambda x,y)=F(x,\lambda^\tau
y)\mbox{\quad for each $x,y\in\F_{p^n}$ and for each $\lambda\in\F_q$\quad }$$ then $\cN_m(\S)$
contains the field of maps $F_q$ and $\cN_m(\S)\subseteq
{\rm End}_{\F_q}(\F_{p^n})$. \item[(B)] If the polynomial $F$ is
$\F_q$-semilinear with respect to $y$,  then $\cN_r(\S)$ contains
the field of maps $F_q$ and $\cN_r(\S)\subseteq
{\rm End}_{\F_q}(\F_{p^n})$. \item[(C)] If the polynomial $F$ is
$\F_q$-semilinear with respect to $x$, then $\cN_l(\S)$ contains
the field of maps $F_q$ and $\cN_l(\S)\subseteq
{\rm End}_{\F_q}(\F_{p^n})$.
\end{itemize}
\end{theorem}
\begin{proof}
Let us prove Statement $(A)$. Recall that $\cN_m(\S)$ is the
largest field contained in $\V={\rm End}_{\F_p}(\F_{p^n})$ such that
$\cC\cN_m(\S)\subseteq \cC$. Let $\lambda\in\F_q$, then, for
each $x,y\in\F_{p^n}$ $$\varphi_y\circ t_\lambda
(x)=\varphi_y(\lambda x)=F(\lambda x,y)=F(x,\lambda^\tau
y)=\varphi_{\lambda^\tau y}(x).$$ This means that for each
$y\in\F_{p^n}$, $\varphi_y\circ t_\lambda=\varphi_{\lambda^\tau
y}\in \cC$, i.e. $\cC\,F_q\subseteq \cC$. Hence $F_q\subseteq
\cN_m(\S)$ and since $(\cN_m(\S),+,\circ)$ is a field, we get
$\mu\circ t_\lambda=t_\lambda\circ\mu$ for each
$\lambda\in\F_q$ and $\mu\in\cN_m(\S)$, i.e. $\mu\in {\rm End}_{\F_q}(\F_{p^n})$.
Using similar arguments we can prove Statement $(B)$.\\

Finally, let $\cC^d$ be the spread set associated with the dual
presemifield $\S^d$ of $\S$ and let $\sigma$ be the  automorphism
of $\F_q$ associated with $F$ with respect to the variable $x$, i.e.
$F(\lambda x,y)=\lambda^\sigma F(x,y)$ for each $x,y\in \F_{p^n}$
and $\lambda \in \F_q$. Then, for each $\lambda\in\F_q$ and for
each $x,y\in\F_{p^n}$ we have
$$t_\lambda\circ\varphi^x(y)=\lambda\varphi^x(y)=\lambda F(x,y)=F(\lambda^{\sigma^{-1}}x,y)=\varphi^{\lambda^{\sigma^{-1}}x}(y).$$
This means that for each $x\in\F_{p^n}$, $t_\lambda\circ
\varphi^x=\varphi^{\lambda^{\sigma^{-1}}x}\in \cC^d$, i.e.
$F_q\cC^d\subseteq \cC^d$. It follows that $F_q\subseteq
\cN_\ell(\S)$ and hence $\cN_l(\S)\subseteq {\rm End}_{\F_q}(\F_{p^n})$.
\end{proof}

\bigskip

\begin{cor} \label{cor:nucleipresemifield}
Let $\S = (\F_{p^{n}},+,\star)$ be a presemifield whose
multiplication is given by $x\star y=F(x,y)$. If $F(x,y)$ is a
$q$--polynomial ($\F_q$ subfield of $\F_{p^{n}}$), i.e. $\cC,\cC^d
\subseteq {\rm End}_{\F_q}(\F_{p^n})$, then
$$F_q \subseteq \cN_l(\S)\cap \cN_m(\S)\cap
\cN_r(\S),$$

$$ F_q \subseteq  \cK_{r,\omega}(\S) \cap \cK_{m, \sigma}(\S)$$
and
$$\cK_{m, \sigma}, \cK_{r, \omega}, \cN_l(\S),\ \cN_m(\S),\ \cN_r(\S)\subseteq {\rm End}_{\F_q}(\F_{p^t})$$

for each $\omega, \sigma \in \cC\setminus \{0\}$.
\end{cor}
\begin{proof}
It is sufficient to note that in this case we can write
$F(x,y)=\sum_{i,j=0}^{h-1}a_{ij}x^{q^i}y^{q^j}$ with
$a_{ij}\in\F_{p^n}$ and $q^h=p^n$. Then for each $\lambda\in\F_q$,
we get $F(\lambda x,y)=\lambda F(x,y)=F(x,\lambda y)$ and hence by
Theorem \ref{thm:nucleipresemifieldfromF} and  point $(d)$ of Theorem
\ref{thm:nucleipresemifield}, the assertion follows.
\end{proof}

Recall that  the dual and the transpose operations are invariant
under isotopy. Hence it makes sense to ask which is the isotopism
involving the duals and the transposes of two isotopic
presemifields (see \cite[Proposition 2.3]{MaPoSub}). Precisely, if
$\langle, \rangle$ is a given non-degenerate symmetric bilinear
form of $\F_{p^n}$ over $\F_p$ and $\overline \varphi$ denotes the
adjoint of $\varphi \in \V$ with respect to $\langle, \rangle$,
then we can prove the following

\begin{prop}\label{prop:DualTranspIsot}
Let $\S_1= (\F_{p^{n}},+,\bullet)$ and $\S_2=
(\F_{p^{n}},+,\star)$ be two presemifields. Then
\begin{itemize}
\item[$i)$] $(g_1,g_2,g_3)$ is an isotopism between $\S_1$ and
$\S_2$ if and only if $(g_2,g_1,g_3)$ is an isotopism between the
dual presemifields $\S_1^d$ and $\S_2^d$;
 \item[$ii)$]  $(g_1,g_2,g_3)$ is an
isotopism between $\S_1$ and $\S_2$ if and only if
$(\overline{g_3}^{^{-1}},g_2,\overline{g_1}^{^{-1}})$ is an
isotopism between the transpose presemifields $\S_1^t$ and
$\S_2^t$.
\end{itemize}
\end{prop}
\begin{proof} Statement $i)$ easily follows from the definition of dual
operation.

Let us prove $ii)$.  Let $\cC_1=\{\varphi_y|\ y\in\F_{p^n}\}$ and
$\cC_2=\{\varphi'_y|\ y\in\F_{p^n}\}$ be the corresponding spread
sets. By the previous arguments the corresponding transpose
presemifields are defined by the following multiplications
$x\bullet^t y=\overline{\varphi_y}(x)$ and $x\star^t
y=\overline{\varphi'_y}(x)$, respectively. The triple
$(g_1,g_2,g_3)$ is an isotopism between $\S_1$ and $\S_2$ if and
only if $g_3\circ\varphi_y=\varphi'_{g_2(y)}\circ g_1$ for each
$y\in\F_{p^n}$. Since $\overline{\varphi_y}\circ
\overline{g_3}=\overline{g_1}\circ \overline{\varphi'_{g_2(y)}}$
for each $y\in\F_{p^n}$, we have
 $$\overline{g_3}(x)\bullet^t
y=\overline{g_1}(x\star^tg_2(y))$$ for each $x,y\in\F_{p^n}$. This
is equivalent to $\overline{g_1}^{^{-1}}(z\bullet^t
y)=\overline{g_3}^{^{-1}}(z)\star^tg_2(y)$ for each
$z,y\in\F_{p^n}$. So, the assertion follows.
\end{proof}

\begin{remark}
{\rm If $\S_1$ and $\S_2$ are two isotopic presemifields, by using
$i)$ and $ii)$ of the previous preposition, it is possible to
determine the isotopisms between the other derivatives of $\S_1$
and $\S_2$. }
\end{remark}

Finally, if two presemifields are both defined by $\F_q$-linear
maps, then we have a restriction on the possible isotopisms
between them (see \cite[Thm. 2.2]{MaPoSub}).

\begin{theorem}\label{thm:semilinearity}
If $(g_1,g_2,g_3)$ is an isotopism between two presemifields
$\S_1$ and $\S_2$ of order $p^n$, whose associated spread sets
$\cC_1$ and $\cC_2$ are contained in ${\rm End}_{\F_q}(\F_{p^n})$
($\F_q$ a subfield of $\F_{p^n}$), then $g_3$ and $g_1$ are
$\F_q$--semilinear maps of $\F_{p^n}$ with the same companion
automorphism.
\end{theorem}
\begin{proof}
Since $\cC_1,\, \cC_2\subset {\rm End}_{\F_q}(\F_{p^n})$, by Corollary
\ref{cor:nucleipresemifield}, we have that
$$F_q \subset \cN_l(\S_1)\cap\cN_l(\S_2).$$
Also by Corollary  \ref{cor:nucleiisotopy},
$\cN_l(\S_2)=g_3\cN_l(S_1)g_3^{-1}$. Then $g_3F_q g_3^{-1}\subset
\cN_l(\S_2)$, and since a finite field contains a unique subfield of
given order, it follows $g_3F_qg_3^{-1}=F_q$. Hence the map
$t_\lambda\mapsto g_3 t_\lambda g_3^{-1}$ is an automorphism of
the field of maps $F_q$, and so there exists $i\in\{0,\dots,k-1\}$ such
that $g_3 t_\lambda g_3^{-1}=t_{\lambda^{p^i}}$ (where $q=p^k$) for each
$\lambda\in\F_q$, i.e. $g_3$ is an $\F_q$--semilinear map of
$\F_{p^n}$ with companion automorphism $\sigma(x)=x^{p^i}$. Finally,
by Proposition \ref{prop:IsotopicSemifAndSpreadSets},
$g_3\cC_1g_1^{-1}=\cC_2$, and hence $g_1$ is an $\F_q$--semilinear
map of $\F_{p^n}$ as well, with the same companion automorphism
$\sigma$.
\end{proof}

\section{The known families of commutative semifields}
In 2003, Kantor in its article \cite{Kantor2003} pointed out as commutative semifields in odd characteristic, in particular when the characteristic is greater than 3, were very rare objects. Indeed until then the known examples of commutative proper \footnote{Here a presemifield is called {\it proper} if it not isotopic to a field} (pre)semifields of {\bf odd order} were
\begin{itemize}
\item[$\cD$)] {\em Dickson semifields} {\rm\cite{Dickson1906}}:\quad $(\F_{q^k}\times\F_{q^k},+,\star)$, $q$ odd and $k>1$ odd, with
$$(a,b)\star (c,d)=(ac+jb^\sigma d^\sigma,ad+bc),$$
where $j$ is a nonsquare in $\F_{q^k}$, $\sigma$ is an $\F_q$--automorphism of $\F_{q^k}$, $\sigma\ne\,id$. These presemifields have middle nucleus of order
$q^k$ and center of order $q$ (see {\rm\cite{Dickson1905}},
{\rm\cite{Dickson1906}}, {\rm\cite{Dickson1906-1}}).
\item[$\cA$)] {\em Generalized twisted fields} {\rm\cite{Albert1961P}}:\quad $(\F_{q^t},+,\star)$, $q$ odd and $t>1$ odd, with
$$x\star y=x^\alpha y+xy^\alpha,$$
where $\alpha:x\mapsto x^{q^n}$ is automorphism of $\F_{q^t}$, with ${\rm Fix}\,\sigma=\F_q$ and $\frac{t}{\gcd(t,n)}$ is odd. These presemifields have middle nucleus and center both of order $q$ {\rm (see \cite{Albert1961A})}.
\item[$\cG$)] {\em Ganley semifields} {\rm\cite{Ganley1981}}:\quad $(\F_{3^r}\times\F_{3^r},+,\star)$, $r\geq 3$ odd, with
$$(a,b)\star (c,d)=(ac-b^9d-bd^9,ad+bc+b^3d^3).$$ These semifields have middle nucleus and center both of order $3$.
\item[$\cCG$)] {\em Cohen--Ganley semifields} {\rm\cite{CoGa1982}}:\quad $(\F_{3^s}\times\F_{3^s},+,\star)$, $s\geq 3$, with
$$(a,b)\star (c,d)=(ac+jbd+j^3(bd)^9, ad+bc+j(bd)^3)$$
where $j$ is a nonsquare in $\F_{3^s}$. These semifields have middle nucleus of order $3^s$ and center of order $3$.
\item[$\cCM/\cDY$)] {\em Coulter--Matthews/Ding--Yuan presemifields} {\rm\cite{CoMa1997}, \cite{DiYu2006}}:\quad $(\F_{3^e},+,\star)$, $e\geq 3$ odd, with
$$x\star y=x^9y+xy^9\pm 2x^3y^3-2xy.$$
Arguing as in the proof of Theorem \ref{thm:nucleiB},
straightforward computations show that the
$\cCM/\cDY$ presemifields have nuclei and center all of order
$3$. In \cite{CoHe2008}, the authors have showed that, for each
$e\geq 5$ odd, these two presemifields are not isotopic and they
are not isotopic to any previously known commutative semifield.
\item[$\cPW/\cBLP$)] {\em Penttila--Williams/Bader--Lunardon--Pinneri
semifield} {\rm\cite{PeWi2000},
\cite{BaLuPi1999}}:\quad $(\F_{3^5}\times\F_{3^5},+,\star)$, with
$$(a,b)\star (c,d)=(ac+(bd)^9,ad+bc+(bd)^{27}).$$
This commutative semifield arises from the symplectic semifield
associated with the Penttila--Williams translation ovoid of
$Q(4,3^5)$. The $\cPW/\cBLP$ semifield has middle nucleus of order
$3^5$ and center of order $3$.
\item[$\cCHK$)] {\em Coulter--Henderson--Kosick presemifield} {\rm\cite{CoHeKo2007}}:\quad $(\F_{3^8},+,\star)$, with
    $$x\star y=xy+L(xy^9+x^9y-xy-x^9y^9)+x^{243}y^3+x^{81}y-x^{9}y+x^3y^{243}+xy^{81}-xy^{9},$$where $L(x)=x^{3^5}+x^{3^2}$. This presemifield has middle nucleus of order $3^2$ and center of order $3$.
\end{itemize}

\bigskip

Note that two (pre)semifields belonging to different families of the previous list are not isotopic.

\bigskip

In the last years some other commutative semifields have been constructed, precisely:
\begin{itemize}
\item[$\cZKW$)] {\em Zha--Kyureghyan--Wang presemifields}
{\rm\cite{ZaKyWa2009}, \cite[Thm. 4]{Bierbrauer2009}}:\quad
$(\F_{p^{3s}},+,\star)$, with
$$x\star y=y^{p^t}x+yx^{p^t}-u^{p^{s}-1}(y^{p^{s+t}}x^{p^{2s}}+y^{p^{2s}}x^{p^{s+t}}),$$
where $u$ is a primitive element of $\F_{p^{3s}}$ and
$0<t<3s$ such that $\frac s{\gcd(s,t)}$ is odd and
\begin{equation}\label{form:cinesi}
\frac s{\gcd(s,t)}+\frac t{\gcd(s,t)}\equiv\,0\,(mod\,3).
\end{equation}
In \cite[Cor. 1]{ZaKyWa2009}, it has been proven that, if $p\geq 5$, $s$ is odd and $t\ne 2s$, these presemifields are not isotopic to any previously know presemifield listed above. In \cite[Cor. 3]{LuMaPoTr2011} the same result has been obtained also when $s$ is even.

In \cite{Bierbrauer2010}, the author has proven that the previous multiplication gives rise to a commutative presemifield if, instead of Condition (\ref{form:cinesi}), the following condition is fulfilled
\begin{equation}\label{form:bierb}
p^s\equiv p^t\equiv\,1\,(mod\,3).
\end{equation}
Moreover, in \cite[Thm. 7]{Bierbrauer2009} it has shown that, when $p\equiv \,1\,(mod\,3)$ these presemifields are not isotopic to a Generalized twisted field.

\item[$\cB$)] {\em Bierbrauer presemifields}
{\rm\cite{Bierbrauer2010}}:\quad \quad $(\F_{p^{4s}},+,\star)$,
$p$ odd prime, with
$$x\star y=y^{p^t}x+yx^{p^t}-u^{p^{s}-1}(y^{p^{s+t}}x^{p^{3s}}+y^{p^{3s}}x^{p^{s+t}}),$$
where $u$ is a primitive element of $\F_{p^{4s}}$ and $0<t<4s$
such that $\frac {2s}{\gcd(2s,t)}$ is odd and $p^s\equiv
p^t\equiv\,1\,(mod\,4)$. In \cite[Thm. 7]{Bierbrauer2010}, it has
been proven that, if $t=2$ and $s>1$, these presemifields are not
isotopic neither to a Dickson semifield nor to a Generalized twisted field.
\end{itemize}

In \cite{BuHe2008}, two families of commutative presemifields of
order $p^{2m}$, $p$ odd prime, are constructed starting from
certain Perfect Nonlinear DO--polynomials over $\F_{p^{2m}}$
labeled as $(i^*)$ and $(ii^*)$. In \cite[Thm. 3]{BuHe2010} it has
been shown that the middle nucleus of the presemifields of type
$(i^*)$ has square order. In this way the authors have proven that
for $p\ne 3$ and $m$ odd the commutative presemifields of type
$(i^*)$ are new (\cite[Cor. 8]{BuHe2010}). Later on, in
\cite{BierbrauerSub}, these presemifields are simplified. More
precisely, the author proves that these two families of
presemifields are contained, up to isotopy, into the following
family

\begin{itemize}
\item[$\cBH$)] {\em Budaghyan--Helleseth presemifields}
{\rm\cite{BuHe2008}, \cite{BierbrauerSub}}:\quad \quad
$(\F_{p^{2m}},+,\star)$, $p$ odd prime and $m>1$, with
\begin{equation}\label{form:multiplSimplBHp}
x\star y=xy^{p^m}+x^{p^m}y+[\beta(xy^{p^s}+x^{p^s}y)+\beta^{p^m}(xy^{p^s}+x^{p^s}y)^{p^m}]\omega,
\end{equation}
where $0<s<2m$, $\omega$ is an element of
$\F_{q^{2m}}\setminus \F_{q^m}$ with $\omega^{q^m}=-\omega$ and the following conditions are satisfied:
\begin{equation}\label{form:condMultSimplBHp}
\beta\in\F_{p^{2m}}^*:\ \ \beta^{\frac{p^{2m}-1}{(p^m+1,p^s+1)}}\ne 1\quad\quad\mbox{and}\quad\quad
\mbox{$\not\!\exists\ a\in\F_{p^{2m}}^*:$ $a+a^{p^m}=a+a^{p^s}=0$.}
\end{equation}
\end{itemize}

Also in \cite{BierbrauerSub}, the author presents the family of
commutative semifields

\begin{itemize}
\item[$\cP$)] {\em $P(q,\ell)$ semifields}
{\cite{BierbrauerSub}}:\quad \quad $(\F_{q^{2\ell}},+,\star)$, $q$
odd prime power and $\ell=2k+1>1$ odd, with
$$x*y=\frac{1}2(xy+x^{q^\ell}y^{q^\ell})+\frac 14G(xy^{q^2}+x^{q^2}y),
$$
where
$G(x)=\sum_{i=1}^{k}(-1)^{i}(x-x^{q^{\ell}})^{q^{2i}}+\sum_{j=1}^{k-1}(-1)^{k+j}(x-x^{q^{\ell}})^{q^{2j+1}}$.

\end{itemize}

These semifields generalize the semifields constructed in
\cite{LuMaPoTr2011}, which have order $q^6$, middle nucleus of
order $q^2$ and center of order $q$ (\cite[Thm. 8]{LuMaPoTr2011}).
In \cite{BierbrauerSub} it has been proven that $\cP$ is not
isotopic to any previously known semifield with the possible
exception of $\cBH$ presemifields. Indeed, it has been recently
proven, in \cite{MaPoSub}, that each $\cP$ semifield is
isotopic to a $\cBH$ presemifield.

\bigskip

The aim of this paper is to study the isotopy relation involving
the commutative presemifields listed above. In order to do this, a
very useful tool will be the computation of the order of their
middle nucleus and their center. (Recall that, if a presemifield
is isotopic to a commutative semifield $\S$, then
$\N_\ell(\S)=\N_r(\S)=\cK(\S)$.)

\section{The isotopy issue}

In this section we want to face with the isotopy issue between the
presemifields listed in the previous section. In order to do this we first
compute the nuclei of the involved presemifields.

\subsection{The nuclei of $\mathcal BH$ presemifields}
Let $p$ be an odd prime, $m$ and $s$ two positive integers such
that $0<s<2m$.  Let $\omega$ be an element of
$\F_{p^{2m}}\setminus \F_{p^m}$ with $\omega^{p^m}=-\omega$, the
Budaghyan--Helleseth presemifields presented in
\cite{BierbrauerSub} are defined by Multiplication
(\ref{form:multiplSimplBHp}) under Conditions
(\ref{form:condMultSimplBHp}).

Set $h:=\gcd(m,s)$, then $m=h\ell$ and $s=hd$, where $\ell$ and
$d$ are two positive integers such that $0<d<2\ell$ and
$\gcd(\ell,d)=1$. Putting $q=p^h$, then
$\omega\in\F_{q^{2\ell}}\setminus \F_{q^\ell}$ and
$\omega^{q^\ell}=-\omega$ and the Budaghyan--Helleseth
presemifields $BH(p,m,s,\beta)$ will be
denoted by $\overline{BH}(q,\ell,d,\beta)=(\F_{q^{2\ell}},+,\star)$. Moreover,
Multiplication (\ref{form:multiplSimplBHp}) and Conditions
(\ref{form:condMultSimplBHp}) can be rewritten as
\begin{equation}\label{form:MultBH}
x\star y=
xy^{q^\ell}+x^{q^\ell}y+[\beta(xy^{q^d}+x^{q^d}y)+\beta^{q^\ell}(xy^{q^d}+x^{q^d}y)^{q^\ell}]\omega,
\end{equation}

where
\begin{equation}\label{form-beta1}
\beta\in\F_{q^{2\ell}}^*:\ \ \beta^{\frac{q^{2\ell}-1}{(q^\ell+1,q^d+1)}}\ne 1,
\end{equation}
and
\begin{equation}\label{form-s1}
\not\!\exists\ a\in\F_{q^{2\ell}}^*:\ \ a+a^{q^\ell}=a+a^{q^d}=0.
\end{equation}

Referring to Multiplication (\ref{form:MultBH}) for the two
families of commutative presemifields of type $(i^*)$ and $(ii^*)$
presented in \cite{BuHe2008}, it has been proven that in both
cases their middle nucleus always contains a field of order $q$
(see \cite[Prop. 5 and Prop. 7]{BuHe2010}). Indeed we will prove that it
has order $q^2$.

In \cite[Sec. 3]{MaPoSub} it has been proved that
(\ref{form-beta1}) and (\ref{form-s1}) are equivalent to
\begin{equation}\label{form-s1N}
\ell+d \ \ \mbox{odd}
\end{equation}
and
\begin{equation}\label{form-beta1N}
\beta \ \ \mbox{nonsquare in}\ \ \F_{q^{2l}}.
\end{equation}

Now we can prove

\begin{theorem}\label{thm:nucleiBH}
A $\overline{BH}(q,\ell,d,\beta)$ presemifield of order
$q^{2\ell}$, $q$ an odd prime power and $\ell>1$, has middle nucleus
of order $q^{2}$ and right nucleus, left nucleus and center all of
order $q$.
\end{theorem}

\begin{proof}
Recall that  $0<d<2\ell$ with $\ell+d$ odd and $\gcd(\ell,d)=1$.
Set $x\circ_r y=xy^{q^r}+x^{q^r}y$ for any integer $0<r<2\ell$,
then
\begin{equation}\label{form:spreadsetBH}
\cC=\{\varphi_y:\ x\mapsto x\circ_\ell
y+[\beta(x\circ_dy)+\beta^{q^\ell}(x\circ_dy)^{q^\ell}]\omega\ |\
y\in\F_{q^{2\ell}}\}
\end{equation}
is the spread set associated with the presemifield
$\overline{BH}(q,\ell,d,\beta)$. In particular $\cC$ is contained
in the vector space $\mathbb
V={\mathrm{End}}_{\F_q}(\F_{q^{2\ell}})$.

By $(b)$ of Theorem \ref{thm:nucleipresemifield} and by $(A)$ of Theorem \ref{thm:nucleipresemifieldfromF}, the middle nucleus
of $\overline{BH}(q,\ell,d,\beta)$ is isomorphic to the largest
field, say $\cN_m(\S)$, contained in the space $\mathbb V$ and
satisfying the property $\varphi_y\circ\psi \in \cC$, for each
$\varphi_y \in \cC$ and for each $\psi\in\cN_m(\S)$. This is
equivalent to say that for each $ x,y\in\F_{q^{2\ell}}$ there
exists an element $z\in\F_{q^{2\ell}}$ such that
$\varphi_y(\psi(x))=\varphi_z(x)$, i.e. there exists
$z\in\F_{q^{2\ell}}$ such that
$$\psi(x)\circ_\ell y+[\beta(\psi(x)\circ_dy)+\beta^{q^\ell}(\psi(x)\circ_dy)^{q^\ell}]\omega=x\circ_\ell z+[\beta(x\circ_dz)+\beta^{q^\ell}(x\circ_dz)^{q^\ell}]\omega$$
for each $x,y\in\F_{q^{2\ell}}$. Since $\{1,\omega\}$ is an $\F_{q^\ell}$--basis of $\F_{q^{2\ell}}$, this is equivalent to show there
exists $z\in\F_{q^{2\ell}}$ such that for each $x,y\in\F_{q^{2\ell}}$
the following system
\begin{equation}\label{systemMiddleNucleus}
\left\{\begin{array}{ll} \psi(x)\circ_\ell y=x\circ_\ell z\\
\beta(\psi(x)\circ_d y)+\beta^{q^\ell}(\psi(x)\circ_d y)^{q^\ell} =
\beta(x\circ_d z)+\beta^{q^\ell}(x\circ_d z)^{q^\ell}
\end{array}
\right.
\end{equation}
admits solutions.
Since $\psi\in \V={\rm End}_{\F_q}(\F_{q^{2\ell}})$, we have $\psi(x)=\sum_{i=0}^{2\ell-1}a_ix^{q^i}$, with
$a_i\in\F_{q^{2\ell}}$ and looking at the first equation of System
(\ref{systemMiddleNucleus}) we get
$$(\sum_ia_ix^{q^i})y^{q^\ell}+(\sum_ia_i^{q^\ell}x^{q^{i+\ell}})y=xz^{q^\ell}+x^{q^\ell}z,$$
for each $x,y\in\F_{q^{2\ell}}$. Hence, reducing the above
equation modulo $x^{q^{2\ell}}-x$, we have that for each
$y\in\F_{q^{2\ell}}$ there exists $z\in\F_{q^{2\ell}}$ such that
$$a_iy^{q^\ell}+a_{i+\ell}^{q^\ell}y=\left\{\begin{array}{ll}0 & \mbox{if $i\ne
0,\ell$}\\
z^{q^\ell} & \mbox{if $i= 0$}\\
z & \mbox{if $i=\ell$}
\end{array}
\right.$$ where $a_i=a_j$ if and only if $i\equiv
j\,(mod\,2\ell)$. From the last equalities we get that $a_i=0$ for
each $i\ne 0,\ell$ and $z=a_0^{q^\ell}y+a_{\ell}y^{q^\ell}$.
Hence, from the first equation of System
(\ref{systemMiddleNucleus}), it follows that if
$\psi\in\cN_m(\S)$, then $\psi(x)=\psi_{A,B}(x)=Ax+Bx^{q^\ell}$,
with $A,B\in\F_{q^{2\ell}}$, and $z=A^{q^\ell}y+By^{q^\ell}$ for
each $y\in\F_{q^{2\ell}}$. Substituting these conditions in the
second equation of (\ref{systemMiddleNucleus}), we get that for
each $x,y\in\F_{q^{2\ell}}$ the following polynomial identity must
be satisfied
$$\beta((Ax+Bx^{q^\ell})\circ_d y)+\beta^{q^\ell}((Ax+Bx^{q^\ell})\circ_d
y)^{q^\ell}=\beta(x\circ_d(A^{q^\ell}y+By^{q^\ell}))+\beta^{q^\ell}(x\circ_d(A^{q^\ell}y+By^{q^\ell}))^{q^\ell}.$$
Reducing modulo $x^{q^{2\ell}}-x$ and equating the
coefficients of the obtained reduced polynomials we have that $A$
and $B$ must verify the system
$$\left\{\begin{array}{l}
\beta A=\beta A^{q^{\ell+d}}\\
\beta B=\beta^{q^\ell}B^{q^{\ell+d}}.
\end{array}
\right.$$ Note that, since $\gcd(2\ell,\ell+d)=\gcd(\ell,d)=1$ the
set of solutions in $\F_{q^{2\ell}}$ of the equations
$x^{q^{\ell+d}-1}=1$ is the set of nonzero elements of $\F_q$.
Hence, taking into account that $\beta\ne 0$ and $(\beta^{q^\ell-1})^{\frac{q^{2\ell}-1}{q-1}}=1$, we get
$$A\in\F_{q}\quad\quad\mbox{and}\quad\quad B=b\xi,$$ where
$b\in\F_{q}$ and $\xi$ is an element of $\F_{q^{2\ell}}$ such
that $\xi^{q^{\ell+d}-1}=\beta^{1-q^\ell}$. It follows that
$$\cN_m(\S)=\{\psi_{a,b\xi}:\,x\mapsto ax+b\xi x^{q^\ell}|\ \ a,b\in\F_{q}\},$$
and hence the middle nucleus of the presemifield $\overline{BH}(q,\ell,d,\beta)$
has order $q^{2}$.

\medskip

On the other hand, by $(a)$ of Theorem \ref{thm:nucleipresemifield} and
by $(B)$ of Theorem \ref{thm:nucleipresemifield}, the
right nucleus of $\overline{BH}(q,\ell,d,\beta)$ is isomorphic to
the largest field, say $\cN_r(\S)$, of the space $\mathbb
V=\mathrm{End}_{\F_q}(\F_{q^{2\ell}})$, whose elements $\phi
\colon x \mapsto \sum_{i=0}^{2\ell-1}a_ix^{q^i}$, with $a_i\in
\F_{q^{2\ell}}$, satisfy the property $\phi\circ\varphi_y \in
\cC$, for each $\varphi_y \in \cC$. Arguing as above we get that
$\cN_r(\S)=\{x\mapsto ax|\ \ a\in\F_{q}\}$. Since a presemifield
$\overline{BH}(q,\ell,d,\beta)$ is commutative, then its left
nucleus and its center have both order $q$.

Now the statement has been completely proven.
\end{proof}

\begin{cor}
Each $\cP$ semifield of order $q^{2\ell}$, $q$ odd prime and
$\ell>1$ odd, has middle nucleus of order $q^2$ and right nucleus,
left nucleus and center all of order $q$.
\end{cor}
\begin{proof}
It follows from Theorem \ref{thm:nucleiBH} and by \cite[Thm.
4.5]{MaPoSub}.
\end{proof}

Recall that in \cite[Cor. 8]{BuHe2010} the authors have proven
that some $\cBH$ presemifields of order $p^{2m}$, $p>3$ odd prime
and $m$ odd are isotopic nor to a Dickson semifield nor to a
Generalized twisted field and, obviously, nor to a presemifiel of
characteristic 3. By using the previous theorem we can now prove a
stronger result.

\begin{cor}\label{cor:BHnew}
A $\overline{BH}(q,2,d,\beta)$ presemifield of order $q^{4}$, $q$ an odd prime power, with Multiplication
(\ref{form:MultBH}), is isotopic to a Dickson semifield.

A $\overline{BH}(q,\ell,d,\beta)$ presemifield of order
$q^{2\ell}$, $q$ an odd prime power and $\ell>2$, with
Multiplication (\ref{form:MultBH}), is isotopic nor to a Dickson
semifield nor to a Generalized twisted field.
\end{cor}
\begin{proof}
If $\ell=2$, from Theorem \ref{thm:nucleiBH} each
$\overline{BH}(q,\ell,d,\beta)$ presemifield is 2--dimensional
over its middle nucleus and 4--dimensional over its center. Hence
it is isotopic to a Dickson semifield (see \cite{Thas1987}, \cite{BlLaBa2003}) .  If
$\ell>2$, by Theorem \ref{thm:nucleiBH}, comparing the dimensions
of the involved presemifields over their middle nucleus and over
their center (see Table \ref{eqtable}), we get the assertion.
\end{proof}

\subsection{The nuclei of Bierbrauer presemifields}

Set $h=\gcd(s,t)$, then $s=hm$ and $t=hn$ with $\gcd(m,n)=1$. Then
the multiplication of a Bierbrauer presemifield
$\B=(\F_{q^{4m}},+,\star)$, $q=p^h$ odd prime power, can be
rewritten as
\begin{equation}\label{form:multiplB}
x\star y=y^{q^n}x+yx^{q^n}-v(y^{q^{m+n}}x^{q^{3m}}+y^{q^{3m}}x^{q^{m+n}}),
\end{equation}
where $v=u^{q^{m}-1}$, $u$ a primitive element of $\F_{q^{4m}}$, and
$0<n<4m$ such that $\frac {2m}{\gcd(2m,n)}$ is odd and $q^m\equiv q^n\equiv\,1\,(mod\,4)$.

\begin{remark}\label{rem:n-even}
{\rm Since $\frac {2m}{\gcd(2m,n)}$ is odd and $\gcd(m,n)=1$, $n$
is even and $m$ is odd. It follows that Condition $q^n\equiv\,1\,(mod\,4)$ is satisfied for each $q$ odd,
whereas Condition $q^m\equiv\,1\,(mod\,4)$ is equivalent to
$q\equiv\,1\,(mod\,4)$.}
\end{remark}

In \cite[Thm. 6]{Bierbrauer2010} it has been proven that a
$\cB$ presemifield of order $q^{4m}$, $q$ odd prime power, has
middle nucleus containing a field of order $q^2$ and center containing $\F_q$.
Moreover, if $q$ is prime, $n=2$ and $m>1$ odd, a
$\cB$ presemifield is not quadratic over its middle nucleus
(\cite[Thm. 7]{Bierbrauer2010}). Finally, if $q=p$ is an odd
prime, $m=3$ and $n=2$, a $\cB$ presemifield has middle nucleus
of order $p^2$ and center of order $p$ (\cite[Thm.
6]{Bierbrauer2009}).

Here we determine the orders of the nuclei and the center of the
involved presemifields, with no restriction on $n$, $m$ or $q$.

\begin{theorem}\label{thm:nucleiB}
A Bierbrauer presemifield $\B=(\F_{q^{4m}},+,\star)$, $q$ an odd
prime power, with Multiplication (\ref{form:multiplB}), has middle
nucleus of order $q^2$ and center of order $q$.
\end{theorem}
\begin{proof}
Let $\cC=\{\varphi_y:\,x\mapsto x\star y| \
y\in\F_{q^{4m}}\}\subset {\mathbb V}={\rm
End}_{\F_q}(\F_{q^{4m}})$ be the spread set associated with the
presemifield $\B$. Note that if $x\star y=F(x,y)$, then $F(\lambda x,y)=F(x,\lambda y)$ for each $\lambda\in\F_{q^2}$. Hence, by $(A)$ of Theorem \ref{thm:nucleipresemifieldfromF}, we get that
the middle nucleus of $\B$ is the largest field $\cN_m(\B)$ contained the field of maps $F_{q^2}$ and
contained in the space ${\rm End}_{\F_{q^2}}(\F_{q^{4m}})$ satisfying the property
$\varphi_y\circ\psi \in\cC$, for each $\varphi_y \in \cC$ and $\psi\in\cN_m(\bB)$. This is
equivalent to say that for each $ x,y\in\F_{q^{4m}}$ there exists
an element $z\in\F_{q^{4m}}$ such that
$\varphi_y(\psi(x))=\varphi_z(x)$, where
$\psi(x)=\sum_{i=0}^{2m-1}a_ix^{q^{2i}}$, with $a_i\in\F_{q^{4m}}$, i.e.
\begin{eqnarray}\label{form:bierb2}
y\sum_ia_i^{q^n}x^{q^{n+{2i}}}+y^{q^n}\sum_ia_ix^{q^{2i}}-v\bigg(y^{q^{m+n}}\sum_ia_i^{q^{3m}}x^{q^{3m+2i}}+y^{q^{3m}}\sum_ia_i^{q^{m+n}}x^{q^{2i+m+n}}\bigg)
=\nonumber\\
=z^{q^n}x+zx^{q^n}-v(z^{q^{m+n}}x^{q^{3m}}+z^{q^{3m}}x^{q^{m+n}}),
\end{eqnarray}
for each $x,y\in\F_{q^{4m}}$.

Now, reduce the above polynomials modulo $x^{q^{4m}}-x$ and equate the
coefficients. Since
the monomials $x,x^{q^n},x^{q^{3m}},x^{q^{m+n}}$ are pairwise
distinct modulo $x^{q^{4m}}-x$, $n+m$ is odd and $\gcd(n,m)=1$, we get
$$
a_{i-\frac n2}^{q^n}y+a_{i}y^{q^n}=
\left\{\begin{array}{lll}
z^{q^n}& \mbox{if $i=0$}& (a)\\
z& \mbox{if $2i=n$} &(b)\\
0 & \mbox{otherwise}
\end{array}
\right.
$$
for all $y\in\F_{q^{4m}}$, where $a_j=a_{j'}$ if and only if
$j\equiv j'(mod\, 4m)$.

\noindent From the previous equalities we get
\begin{equation}\label{form:ai}
a_i=a_{i-\frac n2}=0\quad
\forall\,i\notin\bigg\{0,\frac n2\bigg\} \mbox{\ \ modulo $4m$}.
\end{equation}
 Combining $(a)$ and $(b)$ we get
\begin{equation}\label{form:ai1}
a_{4m-\frac n2}^{q^{n}}y+a_0y^{q^{n}}=a_0^{q^{2n}}y^{q^n}+a_{\frac n2}^{q^n}y^{q^{2n}}.
\end{equation}
Hence
\begin{equation}\label{form:bierb3}
a_{\frac n2}=a_{4m-\frac n2}=0
\end{equation}
and
\begin{equation}\label{form:bierb4}
a_0=a_0^{q^{2n}}\quad\mbox{and}\quad z=a_0^{q^n}y.
\end{equation}

By using (\ref{form:ai}), (\ref{form:bierb3}) and (\ref{form:bierb4}) in Equation (\ref{form:bierb2}), we get that for each $x,y\in\F_{q^{4m}}$ there exists $z\in\F_{q^{4m}}$ such that
$$
a_0y^{q^{n}}x+a_0^{q^{n}}yx^{q^n}-v(a_0^{q^{3m}}y^{q^{m+n}}x^{q^{3m}}+a^{q^{m+n}}y^{q^{3m}}x^{q^{m+n}})=$$
$$=
a_0^{q^{2n}}y^{q^{n}}x+a_0^{q^n}yx^{q^n}-v(a_0^{q^{m+2n}}y^{q^{m+n}}x^{q^{3m}}+a_0^{q^{3m+n}}y^{q^{3m}}x^{q^{m+n}}),
$$
which implies, taking into account
that $y$, $y^{q^{n}}$, $y^{q^{m+n}}$ and $y^{q^{3m}}$
are pairwise distinct modulo $y^{q^{4m}}-y$, that
$$a_0=a_0^{q^{2m}}.$$
From the last equality and from (\ref{form:bierb4}), since $\gcd(n,m)=1$, it follows
$\cN_m(\B)=F_{q^2}$.

\medskip

On the other hand, the right nucleus of $\B$ is the largest field
$\cN_r(\B)$ of the space $\mathbb
V=\mathrm{End}_{\F_q}(\F_{q^{4m}})$, whose elements $\phi \colon x
\mapsto \sum_{i=0}^{4m-1}a_ix^{q^i}$, with $a_i\in \F_{q^{4m}}$,
satisfy the property $\phi\circ\varphi_y \in \cC$, for each
$\varphi_y \in \cC$. Arguing as above we get $\cN_r(\B)=F_{q}$.

\end{proof}

In \cite[Thm. 7]{Bierbrauer2010}, it has been shown that a
$\cB$ presemifield of order $q^{4m}$, $q=p$ an odd prime, $n=2$
and $m>1$ odd, is not isotopic to a Generalized twisted
field. By using the previous theorem we can now prove the
following

\begin{cor}\label{cor:Bnew}
A Bierbrauer presemifield of order $q^4$ (i.e., $m=1$), $q$ an odd
prime power, with Multiplication (\ref{form:multiplB}) is isotopic
to a Dickson semifield.

A Bierbrauer presemifield of order $q^{4m}$, $q$ an odd prime
power and $m>1$ odd, with Multiplication (\ref{form:multiplB}), is
isotopic neither to a Dickson semifield, nor to a Generalized
twisted field and to any of the known commutative presemifields of characteristic 3 (see Table \ref{eqtable}).
\end{cor}
\begin{proof}
From Theorem \ref{thm:nucleiBH} a $\cB$ presemifield of order
$q^4$ is 2--dimensional over its middle nucleus and 4--dimensional
over its center. Hence it is isotopic to a
Dickson semifield (see \cite{Thas1987} and \cite{BlLaBa2003}).  If $m>1$ odd, by Theorem \ref{thm:nucleiBH},
comparing the orders and the parameters of the involved presemifields over their
middle nucleus and over their center (see Table \ref{eqtable}), we get the
assertion.
\end{proof}

\begin{cor}\label{cor:Bnew1}
A $\overline{BH}(q,\ell,d,\beta)$ presemifield of order
$q^{2\ell}$, $q$ an odd prime power and $\ell\ne 2k$ with $k$ odd,
defined by Multiplication (\ref{form:MultBH}), is not isotopic to
any Bierbrauer presemifield with Multiplication
(\ref{form:multiplB}).
\end{cor}
\begin{proof}
The assertion again follows by comparing the dimensions of the involved
presemifields over their middle nucleus and over their center (see Table \ref{eqtable}).
\end{proof}

\subsection{The nuclei of Zha--Kyureghyan--Wang  presemifields}

Set $g=\gcd(s,t)$, then $s=hg$ and $t=ng$ with $\gcd(h,n)=1$. Then the multiplication of a Zha--Kyureghyan--Wang  presemifield $\mathbb{ZKW}=(\F_{q^{3h}},+,*)$, $q=p^g$ odd prime power, can be rewritten as
\begin{equation}\label{form:multiplZKW}
x\star y=y^{q^n}x+yx^{q^n}-v(y^{q^{h+n}}x^{q^{2h}}+y^{q^{2h}}x^{q^{h+n}}),
\end{equation}
where $v=u^{q^{h}-1}$, $u$ a primitive element of $\F_{q^{3h}}$, and
$0<n<3h$ such that $h$ is odd. Hence, by Corollary \ref{cor:nucleipresemifield}, all the nuclei and the center of a $\cZKW$ presemifield contain a field of order $q$. Moreover this multiplication gives rise to a $\cZKW$ presemifield if either
\begin{equation}\label{form:cinesi1}
h+n\equiv\,0\,(mod\,3)
\end{equation}
or
\begin{equation}\label{form:bierb1}
q\equiv\,1\,(mod\,3).
\end{equation}

If $h=1$ and $n=2$, by the form of Multiplication (\ref{form:multiplZKW}) it is clear that a $\cZKW$ presemifield of order $q^3$ is isotopic to a Generalized twisted field.

If $h=n=1$ only Condition (\ref{form:bierb1}) can be realized. Arguing as in Theorem \ref{thm:nucleiB} and taking into account that $v^{q^2+q+1}=1$ and that $q\equiv\,1\,(mod\,3)$, it can be proven that in this case the nuclei and the center have all order exactly $q$. Hence also in this case, by the classification result of Menichetti \cite{Menichetti1978}, the $\cZKW$ presemifield is isotopic to a Generalized twisted field.

More generally, in \cite[Thm. 10]{LuMaPoTr2011}, using an isotopy form, it has been proven that a $\cZKW$ presemifield of order $q^{3h}$, $h>1$ odd, satisfying Condition (\ref{form:cinesi1}) has middle nucleus of order $q$. Using similar techniques as in Theorem \ref{thm:nucleiB} it can be proven that also the center of a $\cZKW$ presemifield has order $q$. In both cases the arguments do not involve the congruences (\ref{form:cinesi1}) and (\ref{form:bierb1}). Hence we obtain

\begin{theorem}\label{thm:nucleiZKW}
A $\cZKW$ presemifield of order $q^{3h}$, $q$ an odd prime power and $h$ an odd integer, with Multiplication (\ref{form:multiplZKW}), has middle nucleus and center both of order $q$. \qed
\end{theorem}

\begin{cor}\label{cor:ZKWnew}
A $\cZKW$ presemifield  of order $q^3$, $q$ an odd prime power, with Multiplication (\ref{form:multiplZKW}), is isotopic to a Generalized twisted field.

A $\cZKW$ presemifield  of order $q^{3h}$, $q>3$ an odd prime power and $h>1$ odd integer, with Multiplication (\ref{form:multiplZKW}), is not isotopic to any known presemifield.
\end{cor}
\begin{proof}
The first part has been proven above. By \cite[Cor. 3]{LuMaPoTr2011}, a $\cZKW$ presemifield of order $q^{3h}$, $q>3$ an odd prime power and $h>1$ odd integer, is isotopic neither to a Dickson semifield nor to a Generalized twisted field and, by Table \ref{eqtable}, it is not isotopic to any presemifield with characteristic $3$. Moreover, by Theorems \ref{thm:nucleiBH}, \ref{thm:nucleiB} and \ref{thm:nucleiZKW}, a $\cZKW$ presemifield is isotopic neither to a $\cBH$ presemifield nor a $\cB$ presemifield by comparing the dimensions of the involved presemifields over their center.
\end{proof}

\bigskip
\bigskip

The following table summarizes the state of the art on the presently known commutative presemifields whose multiplication are written pointing out their center. In this table we have also written the multiplication and the parameters of some presemifields very recently constructed in \cite{ZhPoSub}.

\begin{landscape}
\begin{table}[ht]
\caption{\bf Commutative proper presemifields of odd characteristic}\label{eqtable}
\renewcommand\arraystretch{1.3}
\noindent\[
\begin{array}{|c|c|c|c|c|c|}
\hline\hline
  \small \emph{TYPE} & \small SIZE & \small |\cK|\ \ |\N_m| &
\small {\emph{MULTIPLICATION}}  & \small \emph{EXIS. RESULTS} & \small \emph{REFER.}  \\
 \hline
 & & &  &  &\vspace*{-.4cm}\\
 \mbox{\small   $\cD$ } & \mbox{\tiny   $q^{2k}$, $k>1$ odd }& \mbox{\tiny $q\quad\quad q^k$} &
\mbox{\tiny $(a,b) \star (c,d)=(ac+jb^\sigma d^\sigma, ad+bc)$} & \mbox{\tiny  $\exists$  $\forall q$ odd,}  & \mbox{\tiny \cite{Dickson1905}, \cite{Dickson1906}, \cite{Dickson1906-1}}\\
& &
 & \mbox{\tiny where $\sigma \neq 1$ $\F_q$--autom. of $\F_{q^k}$ and $j$ nonsquare in $\F_{q^k}$} & &\\
 \hline
& & &  & & \vspace*{-.4cm}\\
 \mbox{\small   $\cA$ } & \mbox{\tiny   $q^t$, $t>1$ odd}   & \mbox{\tiny $q$\quad\quad $q$} &
\mbox{\tiny $x\star y= x^\alpha y+x y^\alpha,$ } & \mbox{\tiny $\exists$   $\forall q$ odd,}  & \mbox{\tiny  \cite{Albert1961P}, \cite{Albert1961A}} \\
& &  & \mbox{\tiny where $\alpha:x\mapsto x^{q^n}$ $\F_q$--autom. of $\F_{q^t}$, $\alpha^2\ne 1$ and $\frac t{\gcd(t,n)}$ odd} &&\\
\hline
& & &  & & \vspace*{-.4cm}\\
\mbox{\small   $\mathcal ZKW$ } & \mbox{\tiny   $q^{3h}$, $h>1$ odd   }& \mbox{\tiny
$q$\quad\quad $q$} &
\mbox{\tiny  $x\star y= x^{q^n} y+x y^{q^n}-u^{q^h-1}(x^{q^{2h}}y^{q^{h+n}}+y^{q^{2h}}x^{q^{h+n}})$} & \mbox{\tiny $\exists$   $\forall q$ odd,  $n+h \equiv 0\,(mod \,3)$}  & \mbox{\tiny  \cite{ZaKyWa2009}, \cite{Bierbrauer2010}, \cite{Bierbrauer2009}} \\
& &  & \mbox{\tiny where $ 0<n<3h$, $(h,n)=1$ and $u$ prim. elem. of $\F_{q^{3h}}$} & \mbox{\tiny $\exists$   $\forall q$ odd :  $q \equiv 1\,(mod \,3)$} & \\
\hline
& & &  & & \vspace*{-.4cm}\\
\mbox{\small   $\mathcal B$ } & \mbox{\tiny   $q^{4m}$, $m>1$ odd}   & \mbox{\tiny $q$\quad\quad
$q^2$} &
\mbox{\tiny $x\star y= x^{q^n} y+x y^{q^n}-u^{q^m-1}(x^{q^{3m}}y^{q^{m+n}}+y^{q^{3m}}x^{q^{m+n}})$} & \mbox{\tiny  $\exists$   $\forall q \equiv 1\,(mod 4)$,}  & \mbox{\tiny  \cite{Bierbrauer2010}, \cite{Bierbrauer2009}} \\
& &  & \mbox{\tiny where $ 0<n<4m$, $n$ even,  $\gcd(m,n)=1$ and $u$ prim. elem. of $\F_{q^{4m}}$} & &\\

 \hline
& & &  & & \vspace*{-.4cm}\\

\mbox{\small   ${\mathcal BH}$ } & \mbox{\tiny $q^{2\ell}$, $\ell>2$}   & \mbox{\tiny $q$\quad\quad
$q^2$} & \mbox{\tiny
$x\star y= xy^{q^\ell}+x^{q^\ell}y+[\beta(xy^{q^d}+x^{q^d}y)+\beta^{q^\ell}(xy^{q^d}+x^{q^d}y)^{q^\ell}]\omega$} & \mbox{\tiny $\exists$   $\forall q$ odd,}  & \mbox{\tiny  \cite{BuHe2010}, \cite{LuMaPoTr2011}, \cite{BierbrauerSub}}\\
& &  & \mbox{\tiny where $ 0<d<2\ell$, $\gcd(\ell,d)=1$, $\ell+d$ odd,  $\beta$ nonsquare of $\F_{q^{2\ell}}$ and $\omega^{q^\ell}=-\omega$} & &\\

 \hline
& & &  & & \vspace*{-.4cm}\\

\mbox{\small   ${\mathcal ZP}$ } & \mbox{\tiny $q^{2\ell}$, $\ell>2$}   & \quad\mbox{\tiny $q$\quad $q^2$ (if $\sigma=1$)} & \mbox{\tiny
$(a,b) \star (c,d)= (ac^{q^n}+a^{q^n}c+\alpha(bd^{q^n}+b^{q^n}d)^\sigma,ad+bc),$} & \mbox{\tiny  $\exists$   $\forall q$ odd,}  & \mbox{\tiny  \cite{ZhPoSub}} \\
& & \quad\mbox{\tiny $q$\quad
$q$ \ \,(if $\sigma\ne\,1$)}& \mbox{\tiny where $\sigma:x\mapsto x^{q^t}$ autom. of $\F_{q^\ell}$, $\gcd(\ell,n,t)=1$, $\frac \ell{\gcd(\ell,n)}$ odd and  $\alpha$ nonsquare of $\F_{q^{\ell}}$}
& &\\
 \hline
& & &  & & \vspace*{-.4cm}\\
\mbox{\small   $\mathcal CG$ } & \mbox{\tiny $3^{2s}$, $s\geq 3$}    & \mbox{\tiny $3$
\quad $3^s$} &
\mbox{\tiny $(a,b) \star (c,d)=(ac+jbd+j^3(bd)^9, ad+bc+j(bd)^3)$ } & \mbox{\tiny }  & \mbox{\tiny \cite{CoGa1982}} \\
& &  & \mbox{\tiny $j$ nonsquare in $\F_{3^s}$}  & &\\
 \hline
& & &  & & \vspace*{-.4cm}\\
\mbox{\small   $\cG$ } & \mbox{\tiny $3^{2r}$, $r\geq 3$ odd} & \mbox{\tiny $3$\quad\quad
$3$} &
\mbox{\tiny $(a,b) \star (c,d)=(ac-b^9d-bd^9, ad+bc+(bd)^3)$} & \mbox{}  & \mbox{\tiny \cite{Ganley1981}} \\
& &  & & &\vspace*{-.4cm}\\

\hline
& & &  & & \vspace*{-.4cm}\\

\mbox{\small   $\cCM/\cDY$ } & \mbox{\tiny $3^{e}$, $e\geq 5$ odd} & \mbox{\tiny
$3$\quad\quad $3$} &
\mbox{\tiny $x \star y=x^9y+xy^9\pm 2x^3y^3-2xy$ } & \mbox{}  & \mbox{\tiny  \cite{CoMa1997}, \cite{DiYu2006}, \cite{CoHe2008}}\\
& &  & & &\vspace*{-.4cm}\\

 \hline

& & &  & & \vspace*{-.4cm}\\
\mbox{\small   $\cPW/\cBLP$} & \mbox{\tiny $3^{10}$},   & \mbox{\tiny $3$\quad\quad $3^5$} &
\mbox{\tiny $(a,b) \star (c,d)=(ac+(bd)^9, ad+bc+(bd)^{27})$ } & \mbox{}  & \mbox{\tiny \cite{PeWi2000}, \cite{BaLuPi1999}}\\
& &  & & &\vspace*{-.4cm}\\

 \hline

& & &  & & \vspace*{-.4cm}\\

\mbox{\small   $\mathcal CHK$ } & \mbox{\tiny $3^{8}$},   & \mbox{\tiny $3$\quad\quad $3^2$} &
\mbox{\tiny $x\star y=xy+L(xy^9+x^9y-xy-x^9y^9)+x^{243}y^3+x^{81}y-x^{9}y+x^3y^{243}+xy^{81}-xy^{9}$} & {}  & \mbox{\tiny \cite{CoHeKo2007}}\\
& &  & \mbox{\tiny where $L(x)=x^{3^5}+x^{3^2}$} & &\\

\hline \hline
\end{array}
\]
\end{table}

\end{landscape}

Some explanatory comments on the above table are needed. The sizes of the middle nucleus and the center of a
commutative (pre)semifield belonging to one of the families $\cD$, $\cA$, $\cZKW$, $\cB$, $\cBH$, $\cCG$, $\cG$, $\cCM/\cDY$, $\cPW/\cBLP$, $\cCHK$ and $\mathcal ZP$ are listed in the middle columns. The fifth column contains the existence results of commutative presemifields belonging to the above families.

\subsection{Final remarks}

By comparing the dimensions of the commutative presemifields over
their middle nucleus and their center, taking into account their
characteristic and using  Corollaries \ref{cor:BHnew},
\ref{cor:Bnew}, \ref{cor:Bnew1} and \ref{cor:ZKWnew}, it follows
that the family of $\cZKW$ presemifields of order $q^{3h}$
(with $h>1$) and the family of $\cBH$ presemifields of order $q^{2\ell}$
(with $\ell>2$) are not isotopic and they are actually new. More
precisely, for $q>3$, these presemifields are not isotopic to any
previously known. Regarding the family of
$\cB$ presemifields, so far it is not still clear whether it
contains new examples of presemifields. Indeed, by Table \ref{eqtable} and by
comparing the parameters, it is not possible to exclude the
possibility that a $\cB$ presemifield turns out to be isotopic
to a $\cBH$ presemifield. As well as, it remains to investigate
whether a $\cCM/\cDY$ presemifield of order $3^n$,
$n\equiv\,0(mod\,3)$ could be isotopic to a $\cZKW$ presemifield
and whether the $\cCHK$ presemifield belongs, up to isotopy, to
the family of $\cBH$ presemifields.

Finally in \cite{ZhPoSub}, the authors presented a family of presemifields (the {\em $\mathcal ZP$ presemifields}) of order $q^{2\ell}$ and center of order $q$, computing their nuclei. Moreover, they show that, when $\sigma=id$, in some cases, the $\cBH$ presemifields are contained, up to isotopisms, in the family of $\mathcal ZP$ presemifields, whereas when $\ell>3$ is odd and $q\equiv 1(mod\,4)$ the latter family contains presemifields which are non isotopic to any previously known. Obviously, it would be interesting to complete the study of the isotopisms between these two last mentioned families of presemifields for each value of $\ell$ (odd or even) and for each odd characteristic.

\bibliographystyle{amsalpha}

\end{document}